\RequirePackage{xstring}
\StrBehind*{\jobname}{.}[\DocOpts]

\wlog{DocOpts: \DocOpts}

\documentclass[\DocOpts]{amsart}

\usepackage{amssymb}
\usepackage{enumerate}
\input xy
\xyoption{all}
\usepackage{ifdraft}
\ifdraft{
\usepackage{lineno}
}{

}

\newtheorem{theorem}{Theorem}[section]
\newtheorem{lemma}[theorem]{Lemma}
\newtheorem{corollary}[theorem]{Corollary}

\newtheorem{proposition}[theorem]{Proposition}

\newtheorem{claim}[theorem]{Claim}
\newtheorem*{theorem*}{Theorem}

\theoremstyle{definition}
\newtheorem{example}[theorem]{Example}
\newtheorem{remark}[theorem]{Remark}
\newtheorem{definition}[theorem]{Definition}

\def\aut{\operatorname{Aut}}
\def\autvar{\operatorname{Aut}_{\operatorname{var}}}
\def\autgrp{\operatorname{Aut}_{\operatorname{grp}}}

\def\fix{\operatorname{Fix}}

\def\acfa{\operatorname{ACFA}}

\def\qftp{\operatorname{qftp}}

\def\Stab{\operatorname{Stab}}

\def\acl{\operatorname{acl}}

\def\alg{\operatorname{alg}}

\def\bir{\operatorname{Bir}}
\def\loc{\operatorname{loc}}

\def\C{\mathcal{C}}

\def\U{\mathcal{U}}

\def\PP{{\mathbb{P}}}

\def\CC{\mathbb{C}}

\renewcommand{\AA}{\mathbb{A}}

\providecommand{\id}{\operatorname{id}}

\providecommand{\pgl}{\operatorname{PGL}}

\newcommand{\dto}{\dashrightarrow}

\def\Ind#1#2{#1\setbox0=\hbox{\(#1x\)}\kern\wd0\hbox to 
0pt{\hss\(#1\mid\)\hss}\lower.9\ht0\hbox to 
0pt{\hss\(#1\smile\)\hss}\kern\wd0}

\def\notind#1#2{#1\setbox0=\hbox{\(#1x\)}\kern\wd0\hbox to 
0pt{\mathchardef\nn=12854\hss\(#1\nn\)\kern1.4\wd0\hss}\hbox to 
0pt{\hss\(#1\mid\)\hss}\lower.9\ht0 \hbox to 
0pt{\hss\(#1\smile\)\hss}\kern\wd0}

\newenvironment{claimproof}[1]{%
\par\vspace{.1in}\noindent\emph{Proof of Claim:}\space#1}{%
\vspace{.1in}\hfill\(\maltese\)}

\date{\today}

\title{Wild automorphisms and compound isotriviality}

\author{Jason Bell}
\address{Jason Bell\\ University of Waterloo\\ Department of Pure Mathematics\\ 200 University Avenue West\\ Waterloo, Ontario \  N2L 3G1\\ Canada}
\email{jpbell@uwaterloo.ca}

\author{Rahim Moosa}
\address{Rahim Moosa\\ University of Waterloo\\ Department of Pure Mathematics\\ 200 University Avenue West\\ Waterloo, Ontario \  N2L 3G1\\ Canada}
\email{rmoosa@uwaterloo.ca}

\thanks{The authors were supported by Discovery Grants from the Natural Sciences and Engineering Research Council of Canada.}

\begin{document}


\begin{abstract}
Inspired by the model theory of difference fields in characteristic zero, a class of automorphisms of an algebraic variety, here called {\em compound fundamental isotrivial}, is introduced.
These are algebraic dynamical systems that are built up via a finite sequence of equivariant fibrations from (possibly nonautonomous) algebraic dynamics which trivialise after base extension over themselves.
Every wild automorphism of an abelian variety is compound fundamental isotrivial.
Conversely, it is shown that the only irreducible projective varieties admitting a wild automorphism that is compound fundamental isotrivial are the abelian varieties.
That is, the wild automorphism conjecture of Reichstein, Rogalski, and Zhang is here proven for compound fundamental isotrivial dynamics.
Along the way, a counterexample to the naive generalisation of the conjecture to the nonautonomous setting of $\sigma$-varieties is provided.
\end{abstract}

\maketitle

\setcounter{tocdepth}{1}
\tableofcontents

\section{Introduction}

\noindent
A central problem in algebraic dynamics is to understand automorphisms of projective varieties with no proper nonempty invariant closed subvarieties.
Such maps are called {\em wild automorphisms}, a notion introduced by Reichstein, Rogalski, and Zhang~\cite{rrz}.
Working over an algebraically closed field~$k$ of characteristic zero, the {\em wild automorphism conjecture} asserts that if a projective variety admits a wild automorphism then it is an abelian variety.
This conjecture was motivated by connections with the structure theory of projectively simple graded algebras, which are a natural ring-theoretic analogue of just infinite groups.
Reichstein, Rogalski, and Zhang~\cite{rrz} established the conjecture in dimension at most two and conjectured that it should hold in all dimensions; they also proved it for varieties of positive Kodaira dimension.
Subsequent work of Oguiso and Zhang~\cite{oz} treated the case of negative Kodaira dimension in dimension three, while Kirson~\cite{kirson} conditionally established the result for threefolds of Kodaira dimension zero, assuming deep conjectures in algebraic geometry. Despite this progress, the higher-dimensional case remains open.
The results obtained so far suggest that wild automorphisms should admit a description via equivariant fibrations, which ultimately reduce to translation-type dynamics on abelian varieties.

In this paper we will settle the conjecture for a special class of automorphisms that we call {\em compound fundamental isotrivial}.
While this will be a strong condition to impose on automorphisms generally, it is one that is implied by the conjecture itself.
We will first discuss independent model-theoretic motivation for considering this class, before giving a more precise definition later in this Introduction, and then a careful treatment in Sections~\ref{sec:fi} and~\ref{sec:ci} below.

The  wild automorphism conjecture is a group existence statement; it concludes that the projective variety~$X$ admits the structure of an algebraic group.
In model theory there are several important contexts in which the existence of a definable group is deduced from geometric information.
One of these is the {\em binding group theorem}, which in the setting of algebraic dynamics was established recently in~\cite{qfint}.
It applies to dynamics that are ``quantifier-free internal to the fixed field", in the terminology of model theory, which corresponds to the geometric property of being {\em isotrivial}; these are dynamics which become trivial after difference field base extension -- we will give a more precise definition shortly.
For isotrivial dynamics, the truth of the wild automorphism conjecture follows from combining the binding group theorem with cases of the conjecture already appearing in~\cite{rrz}.
But the class we are interested in here is much wider, it is what a model theorist might call ``quantifier-free analysable in the fixed field", a notion that is so far only implicit in the literature, but for which the condition of {\em compound isotriviality} that we introduce here is a geometric articulation.
These are dynamics that are built up via finitely many equivariant fibrations from isotrivial dynamics -- again, more details to come.
The relevance of this class comes from the manifestation of the Zilber dichotomy in algebraic dynamics, a deep model-theoretic result of Chatzidakis and Hrushovski~\cite{acfa} which, along with its refinements, has played a role in most applications of model theory to algebraic dynamics, including work of  Chatzidakis-Hrushovski~\cite{acfa, ch1, ch2}, Hrushovski~\cite{mm}, Medvedev-Scanlon~\cite{alicetom}, and more recently Kamensky-Moosa~\cite{qfint, qfnmdeg}.
The Zilber dichotomy plays no explicit role here except to suggest that the compound isotrivial case is worth considering.
Indeed, a consequence of the dichotomy, roughly speaking, is that if an algebraic dynamics is not compound isotrivial then it will admit a positive dimensional equivariant image that is sparse in the sense that its cartesian powers do not admit rich invariant families of subvarieties.
See~\cite[Section~4]{qfnmdeg} for a precise geometric articulation of the Zilber dichotomy.

What we propose here is a strategy, motivated by model theory and powered by the binding group theorem, for attacking the wild automorphism conjecture in the compound isotrivial case.
What we actually accomplish, however, is to carry out this strategy in a more restricted class of automorphisms, those that are compound {\em fundamental} isotrivial.
This restriction is at least partially justified by the observation (Corollary~\ref{cor:avwild-ci}) that wild automorphisms of abelian varieties are compound fundamental isotrivial.

Let us now try to be a bit more precise so that we can state our main theorem.
First of all, our point of view will involve working in the more robust category of {\em $\sigma$-varieties} where the dynamics may be twisted by an automorphism of the field of definition.
That is, one fixes a field~$K$, and an automorphism~$\sigma:K\to K$, and then considers pairs $(X,\phi)$ where~$X$ is a variety defined over~$K$ and $\phi:X\to X^\sigma$ is an isomorphism to the transform of~$X$ with respect to the action of~$\sigma$ on~$K$.
An {\em invariant} subvariety is then a subvariety~$Y$ such that $\phi(Y)\subseteq Y^\sigma$, and wildness in this setting makes sense too.
While we are primarily interested in the {\em autonomous} case, when~$X$ and~$\phi$ are defined over~$k\subseteq\fix(K,\sigma)$, so that~$\phi$ is an automorphism of~$X$, we are forced to consider nonautonomous $\sigma$-varieties because they arise naturally as the generic fibres of equivariant maps.
Indeed, if $f:(X,\phi)\to(Y,\psi)$ is an equivariant map of autonomous $\sigma$-varieties over~$k$, then the generic fibre is a (usually nonautonomous) $\sigma$-variety over the difference field $(k(Y),\sigma_\psi)$ where $\sigma_\psi$ is the automorphism of the function field of~$Y$ induced by~$\psi$.
And studying the generic fibres of equivariant maps is central to the approach that we are here proposing.

A $\sigma$-variety $(X,\phi)$ over $(K,\sigma)$ is {\em isotrivial} if there is some difference field extension $(L,\sigma)\supseteq(K,\sigma)$ such that,  over $(L,\sigma)$,  the $\sigma$-variety $(X_L,\phi_L)$ is birationally equivalent to some trivial dynamics $(Y,\id)$ where~$Y$ is a variety defined over $\fix(L,\sigma)$.
The isotriviality is deemed {\em fundamental} if $(L,\sigma)$ can be taken to be $(K(X),\sigma_\phi)$ itself.
Autonomous examples include translations on algebraic groups, see Example~\ref{ex:translation} below.
In fact, a consequence of the binding group theorem of~\cite{qfint}, but also of earlier work in~\cite{bms}, is that all autonomous isotrivial dynamics are similar to this example: the dynamics always comes from an algebraic group action.
The wild automorphism conjecture for such dynamics was already established in~\cite{rrz}.
A first guess might be that the same holds also in the nonautonomous case, but in~$\S$\ref{example} below we prove that
$\begin{bmatrix}
0&1\\
1&t
\end{bmatrix}
$
gives a wild and isotrivial $\sigma$-variety structure on the projective line, over $\mathbb C(t)$ equipped with the shift operator $t\mapsto t+1$.
It is our inability to rule out such examples that forces us to restrict further to the fundamental isotrivial case (which this example is not).

Back in the autonomous setting, we say that an automorphism~$\phi$ of an algebraic variety~$X$ over~$k$ is {\em compound isotrivial}  if there is a sequence of equivariant maps
$\xymatrix{
(X,\phi)\ar[r]^{f_1\ \ } & (X_1,\phi_1)\ar[r]^{\ \ f_2} & \cdots\ar[r]^{f_n\ \ \ \ } & (X_n,\phi_n)
}$
such that $(X_n,\phi_n)$ is (autonomous) isotrivial, and the generic fibre of each~$f_i$ is (nonautonomous) isotrivial.
If all the isotriviality is fundamental then we say that $(X,\phi)$ is {\em compound fundamental isotrivial}.
As we have mentioned, examples include all wild automorphisms of abelian varieties.
We prove:
\begin{theorem*}
Suppose the characteristic is zero and~$X$ is an irreducible projective variety with $\phi\in\aut(X)$ such that $(X,\phi)$ is compound fundamental isotrivial.
If $\phi$ is wild then $X$ is an abelian variety.
\end{theorem*}

This appears as Theorem~\ref{wild-compoundiso} below.
The main tool used to prove it is the binding group theorem for isotrivial $\sigma$-varieties obtained in~\cite{qfint}, which uses in an essential way the model-theoretic study of the quantifier-free fragment of the first-order theory of difference-closed fields.
All the model theory is hidden there; it plays no explicit role in this paper, except possibly that of providing a convenient universal domain for difference-algebraic geometry.
A secondary ingredient is a study of (pre-)homogeneity in the category of $\sigma$-varieties, including an improvement to Weil regularisation under the wildness hypothesis, appearing as Proposition~\ref{prop:wildreg} below.

\medskip

Finally, let us mention that it is natural to formulate an analogue of the wild automorphism conjecture for algebraic vector fields on projective varieties, a much more tractable problem that was solved in~\cite{bims}.

\medskip

\subsection*{Acknowledgements}
We are grateful to the anonymous referee whose comments led to a much cleaner statement of Proposition~\ref{prop:wildreg} below.
We would also like to thank Moshe Kamensky for useful discussions.

\bigskip
\section{Preliminaries on $\sigma$-varieties}

\noindent
We begin by recalling and fixing terminology around possibly nonautonomous algebraic dynamics.
Further details can be found in~\cite[Section~3]{qfint}.
We work throughout over a given inversive difference field of characteristic zero, namely a field~$K$ of characteristic zero equipped with an automorphism~$\sigma$.

\begin{definition}
By an {\em inversive $\sigma$-variety}, $(X,\phi)$ over $(K,\sigma)$, we will mean an absolutely irreducible algebraic variety~$X$ over~$K$ equipped with an isomorphism $\phi:X\to X^\sigma$, where $X^\sigma$ is the transform of~$X$ with respect to the action of $\sigma$ on~$K$.


An {\em invariant subvariety of $(X,\phi)$} is a (not necessarily irreducible) subvariety $Y\subseteq X$ over~$K$ such that $\phi(Y)\subseteq Y^\sigma$.
Note that if $Y$ is absolutely irreducible, then $(Y,\phi|_Y)$ is an inversive $\sigma$-variety in its own right.

An {\em equivariant rational map}, $f:(X,\phi)\dto(Y,\psi)$, between inversive $\sigma$-varieties over $(K,\sigma)$, is a dominant rational map $f:X\dto Y$ such that $\psi\circ f=f^\sigma\circ\phi$ as rational maps from~$X$ to $Y^\sigma$.
It is a {\em morphism} (respectively {\em isomorphism}) if the underlying rational map $f:X\to Y$ is regular (respectively biregular).
\end{definition}

\begin{remark}
\begin{itemize}
\item[(a)]
In the literature it is often only required that $\phi$ be birational, but we restrict our attention to isomorphisms in this paper.
\item[(b)]
As all our $\sigma$-varieties will be inversive, we will often drop this adjective.
\item[(c)]
These notions really depend on the base difference field $(K,\sigma)$.
For example, if $X^\sigma=X$ then $\phi$ is an automorphism of~$X$ and we can view $(X,\phi)$ as a $\sigma$-variety over $(K,\sigma)$ or over $(K,\id)$, and this choice affects what one means by ``invariant" subvariety or ``equivariant" map.
\end{itemize}
\end{remark}

It is convenient to embed $(K,\sigma)$ in a sufficiently saturated difference-closed field  extension $(\U,\sigma)$.
This means that every (possibly infinite) system of $\sigma$-polynomial equations and inequations over a difference subfield of cardinality strictly smaller than $|\U|$ (called {\em small}), with the property that all finite subsystems have solutions in some difference field extension of $(\U,\sigma)$, has a solution in $(\U,\sigma)$.
The existence of such universal domains is at the very beginning of the model theory of difference fields, as carried out in~\cite{acfa}.
We follow the standard convention that all difference fields which arise are implicitly assumed to be small difference subfields of $(\U,\sigma)$, unless explicitly stated otherwise.

A key advantage of working in a universal domain is that, given a $\sigma$-variety $(X,\phi)$ over $(K,\sigma)$, the set
$$(X,\phi)^\sharp:=\{a\in X(\U):\sigma(a)=\phi(a)\}$$
is Zariski dense in~$X$.

\begin{definition}
Suppose $(X,\phi)$ is an inversive $\sigma$-variety over $(K,\sigma)$.

By a {\em generic point of $(X,\phi)$} we mean an element $a\in (X,\phi)^\sharp$ that is Zariski generic in~$X$ over~$K$.

Suppose $f:(X,\phi)\dto(Y,\psi)$ is an equivariant rational map to another inversive $\sigma$-variety over $(K,\sigma)$, such that the general fibres of~$f:X\dto Y$ are absolutely irreducible. 
By a {\em generic fibre} of $f:(X,\phi)\dto(Y,\psi)$ we mean the (closed) fibre $X_b$ of $f:X\dto Y$ over a generic point~$b$ of $(Y,\psi)$ over $(K,\sigma)$, 
viewed as a subvariety of $X_L$ where $L:=K(b)^{\alg}$,
and equipped with the restriction of $\phi_L$ to $X_b$ so that $(X_b,\phi_L)$ is an inversive $\sigma$-variety over $(L,\sigma)$ in its own right.
\end{definition}

\begin{remark}
There are several things to check to make sure that the notion of generic fibre is well defined.
First of all, note that $K(b)$ is itself a difference subfield of $(\U,\sigma)$, because $\sigma(b)=\psi(b)$.
Moreover, using also that~$K$ is inversive, we have that $L=K(b)^{\alg}$ is an inversive difference subfield; see, for example, the proof of Lemma~2.3 of~\cite{qfint}.
Also, our assumption on~$f$ ensures that $X_b$ is absolutely irreducible. 
We claim that $X_b$ is invariant for $(X_L,\phi_L)$ over $(L,\sigma)$.
Indeed, there is a Zariski dense subset $D\subseteq X_b(\U)$ on which~$f$ is defined and of constant value~$b$, so that,
by equivariance,
$$f^\sigma(\phi(D))=\psi(f(D))=\psi(b)=\sigma(b).$$
That is, $\phi(D)\subseteq X^\sigma_{\sigma(b)}=(X_b)^\sigma$.
By Zariski-denseness of~$D$ we have $\phi(X_b)\subseteq(X_b)^\sigma$.
So it does make sense to restrict $\phi_L$ to $X_b$ and obtain thereby an inversive $\sigma$-variety over $(L,\sigma)$.
Note also that generic fibres are unique up to conjugation: if $b'$ is another generic point of $(Y,\psi)$ then $(K(b'),\sigma)$ is isomorphic to $(K(b),\sigma)$, and the two generic fibres $(X_b,\phi|_{X_b})$ and $(X_{b'},\phi|_{X_{b'}})$ are conjugate under this isomorphism.
\end{remark}

For every $n\geq 1$, $(\U,\sigma^n)$ is again a saturated difference-closed field that will serve as a universal domain for $\sigma^n$-varieties.
And to every $\sigma$-variety we can associate a natural $\sigma^n$-variety by iterating as follows:

\begin{definition}
If $(X,\phi)$ is a $\sigma$-variety over $(K,\sigma)$ 
then we obtain a $\sigma^n$-variety $(X,\phi^{(n)})$ over $(K,\sigma^n)$, for each $n\geq 1$, where $\phi^{(n)}:X\to X^{\sigma^n}$ is
$\phi^{\sigma^{n-1}}\circ\cdots\circ\phi^\sigma\circ\phi$.
\end{definition}

We end this section by briefly recalling isotriviality and binding groups.
Isotriviality was defined in the Introduction, but here is an equivalent formulation:

\begin{definition}
A $\sigma$-variety $(X,\phi)$ over $(K,\sigma)$ is {\em isotrivial} if there is a difference field extension $(L,\sigma)\supseteq(K,\sigma)$, a generic point~$a$ of $(X_L,\phi_L)$, and a finite tuple~$c$ from $\fix(\U,\sigma)$, such that $L(a)=L(c)$.
\end{definition}

The notion was formalised in~\cite[$\S$3.2]{qfint}, where various properties and equivalent characterisations were established.
In particular, it is a fact that in the definition of isotriviality we can always take~$L$ to be a function field over~$K$.

Let us treat the example mentioned in the Introduction in some more detail.

\begin{example}[A nontrivial isotrivial automorphism]
\label{ex:translation}
Suppose $G$ is a connected algebraic group over~$K$, $b\in G(K)$, and $\phi:G\to G$ is left translation by~$b$, namely  $g\mapsto bg$.
Then $(G,\phi)$ is isotrivial as a $\sigma$-variety over $(K,\id)$.
Indeed, fix any $h\in (G,\phi)^\sharp$ and let $L=K(h)$.
So $(L,\sigma)$ is a difference field and $K\subseteq\fix(L,\sigma)$.
Let $\rho:G_L\to G_L$ be left translation by $h^{-1}$.
Note that $(G_L)^\sigma=G_L$ and that $\rho^\sigma:G_L\to G_L$ is left translation by $\sigma(h)^{-1}=\phi(h)^{-1}=h^{-1}b^{-1}$.
Hence $\rho^\sigma\circ\phi_L=\rho$, witnessing that $\rho:(G_L,\phi_L)\to(G_L,\id)$ is an isomorphism of $\sigma$-varieties over $(L,\sigma)$.
This witnesses isotriviality: if~$a$ is generic for $(G_L,\phi_L)$ then $c:=\rho(a)\in\fix(\U,\sigma)$ and $L(a)=L(c)$. 
\end{example}

The main result of~\cite{qfint} is that isotriviality gives rise to a {\em binding group}.
That is, an algebraic group~$G$ over~$K$ of birational transformations of~$X$ capturing precisely those birational transformations that preserve all irreducible invariant subvarieties of $\sigma$-varieties of the form $(X^r,\phi)\times(V,\id)$, where~$G$ acts diagonally on $X^r$ and trivially on~$V$.
Moreover, $G$ comes equipped with an isomorphism of algebraic groups $\rho:G\to G^\sigma$ (so that $(G,\rho)$ is a {\em $\sigma$-group}) such that the action of~$G$ on~$X$ is an equivariant rational map $(G,\rho)\times(X,\phi)\dto(X,\phi)$.
In the autonomous case this binding group action can be used to show that isotrivial automorphisms are translational, so that Example~\ref{ex:translation} is not so far from the general case.
But nonautonomous isotrivial $\sigma$-varieties can have a lot more (in particular noncommutative) structure, as is witnessed by the example studied in $\S$\ref{example} below.

\bigskip
\section{Nonautonomous wildness}
\label{sec-sigmavar}

\noindent
Our goal in this section is to extend the notion of wild automorphism to the twisted case of $\sigma$-varieties, to prove some basic lemmas about wild $\sigma$-varieties that we will need later, and to observe by a counterexample that the analogue of the wild automorphism conjecture fails in this nonautonomous setting.

We continue to work over a characteristic~$0$ inversive difference field $(K,\sigma)$.

\begin{definition}
An inversive $\sigma$-variety $(X,\phi)$ over $(K,\sigma)$ is {\em wild} if it has no proper invariant subvarieties.
That is, the only (possibly reducible) subvariety $Y\subseteq X$ with the property that $\phi(Y)\subseteq Y^\sigma$ is $Y=X$.
\end{definition}

One consequence of wildness is that all $\sharp$-points are generic:

\begin{lemma}
\label{isolate}
If $(X,\phi)$ is wild then all points of $(X,\phi)^\sharp$ are generic.
That is, if $a\in (X,\phi)^\sharp$ then $a$ is Zariski generic in $X$ over~$K$.
\end{lemma}

\begin{proof}
This is because $\loc(a/K)$ is an invariant subvariety of $(X,\phi)$.
\end{proof}

The converse doesn't hold: all sharp-points being generic only implies that there are no proper {\em irreducible} invariant subvarieties.

\begin{lemma}
\label{nirfs}
If $(X,\phi)$ is wild then $(X,\phi)$ admits no nonconstant invariant rational functions.
\end{lemma}

\begin{proof}
An invariant rational function is an equivariant rational map
$$f:(X,\phi)\dto (\AA_K^1,\id).$$
Let $a\in \fix(K,\sigma)$ be in the image of~$f$, $X_a$ the closed fibre of~$f$ above~$a$, and $b\in X_a$.
Then $f^\sigma\phi(b)=f(b)=a$, so that $\phi(b)\in X^\sigma_a=X^\sigma_{\sigma(a)}=(X_a)^\sigma$.
That is, $X_a$ is invariant for $(X,\phi)$.
Wildness forces $X_a=X$.
So~$f$ is constant.
\end{proof}

Wildness is preserved by iteration:

\begin{lemma}
\label{iterate-wild}
$\phi$ is wild if and only if $\phi^{(n)}$ is wild.
\end{lemma}

\begin{proof}
As every $\phi$-invariant set is $\phi^{(n)}$-invariant, the right-to-left direction is clear.
Suppose $Z$ is a proper invariant subvariety of the inversive $\sigma^n$-variety $(X,\phi^{(n)})$.
For each $i=0,\dots, n$, let $Z_i:=\big(\phi^{(i)}(Z)\big)^{\sigma^{-i}}$, a proper subvariety of~$X$ over~$K$.
Here $\phi^{(0)}:=\id$ so that $Z_0=Z$.
Note that, by $\phi^{(n)}$-invariance, $Z_n=Z$ also.
Also, for each $i<n$, we have $\phi(Z_i)=Z_{i+1}^\sigma$.
Indeed,
\begin{eqnarray*}
Z_{i+1}^\sigma
&=&
\big(\phi^{(i+1)}(Z)\big)^{\sigma^{-i}}\\
&=&
\big(\phi^{\sigma^i}\circ\phi^{(i)}(Z)\big)^{\sigma^{-i}}\\
&=&
\phi\big((\phi^{(i)}(Z))^{\sigma^{-i}}\big)\\
&=&
\phi(Z_i).
\end{eqnarray*}
Let $\displaystyle \widehat Z:=\bigcup_{i=0}^{n-1}Z_i$.
Then $\widehat Z$ is a proper subvariety of~$X$ and
$\displaystyle \phi(\widehat Z)=\bigcup_{i=0}^{n-1}Z_{i+1}^{\sigma}=\widehat Z^\sigma$, witnessing that $\phi$ is not wild.
\end{proof}

\begin{lemma}
\label{emptyint}
Suppose $(X,\phi)$ is wild.
For any proper subvariety~$Z$ of~$X$,
$$\bigcap_{n\geq 0}(\phi^{(n)})^{-1}(Z^{\sigma^n})=\emptyset.$$
\end{lemma}

\begin{proof}
Note that, by noetherianity, $\displaystyle I:=\bigcap_{n\geq 0}(\phi^{(n)})^{-1}(Z^{\sigma^n})$ is a subvariety of~$X$.
We will show that if~$I$ is nonempty then it is $\phi$-invariant, contradicting wildness.
Given $b\in I$, note that
$(\phi^{(n)})^\sigma\phi(b)=\phi^{(n+1)}(b)\in Z^{\sigma^{n+1}}$, for each $n\geq 0$.
So
$\displaystyle \phi(b)\in  \bigcap_{n\geq 0}((\phi^{(n)})^\sigma)^{-1}(Z^{\sigma^{n+1}})=\big(\bigcap_{n\geq 0}(\phi^{(n)})^{-1}(Z^{\sigma^n})\big)^\sigma=I^\sigma$.
\end{proof}

It follows that an equivariant rational map from a wild inversive $\sigma$-variety to some other inversive $\sigma$-variety will always be regular:

\begin{corollary}
\label{cor:emptyint}
Suppose $f:(X,\phi)\dto(Y,\psi)$ is an equivariant dominant rational map between inversive $\sigma$-varieties over $(K,\sigma)$.
If $(X,\phi)$ is wild then~$f$ is regular.
\end{corollary}

\begin{proof}
Let~$I\subset X$ be the indeterminacy locus of~$f$.
By Lemma~\ref{emptyint},
$$\displaystyle \bigcap_{n\geq 0}(\phi^{(n)})^{-1}(I^{\sigma^n})=\emptyset.$$
Hence, given $b\in X$, there is $n\geq 0$ such that $\phi^{(n)}(b)\notin I^{\sigma^n}$.
It follows that $f^{\sigma^n}$ is defined at $\phi^{(n)}(b)$.
But $f:(X,\phi^{(n)})\dto(Y,\psi^{(n)})$ is an equivariant rational map of inversive $\sigma^n$-varieties, and hence
$$\xymatrix{
X\ar[rr]^{\phi^{(n)}}\ar@{-->}[d]_{f} && X^{\sigma^n}\ar@{-->}[d]^{f^{\sigma^n}}\\
Y\ar[rr]_{\psi^{(n)}} && Y^{\sigma^n}
}$$
commutes.
So $f=(\psi^{(n)})^{-1}f^{\sigma^n}\phi^{(n)}$ is defined at $b$.
That is, $f$ is regular.
\end{proof}

\bigskip
\subsection{A counterexample}
\label{example}
We now exhibit a wild (nonautonomous) inversive $\sigma$-variety structure on the projective line.
In particular, we cannot expect the naive analogue of the wild automorphism conjecture to hold of $\sigma$-varieties.

Our construction will be over the difference field $(K,\sigma)$ where $K:=\CC(t)^{\alg}$ and $\sigma$ is an automorphism satisfying $\sigma(t)=t+1$.
We first observe:

\begin{lemma}
\label{rat}
If $u\in\CC(t)^{\alg}$ is such that $\sigma^n(u)\in\CC(t,u)$ for some $n\geq 1$ then $u\in \mathbb{C}(t)$.
\end{lemma}

\begin{proof}
The assumptions imply that $\sigma^n$ restricts to an infinite-order $\mathbb{C}$-algebra automorphism of the field $L = \mathbb{C}(t, u)$. Let $X$ be the smooth projective curve with function field $\mathbb{C}(X) = L$. Since $\sigma^n$ induces an infinite order automorphism $\tau$ of $X$, $X$ must have genus zero or one. 

We rule out the case of genus one.
The inclusion $\mathbb{C}(t)\subseteq L$ also gives us a $\tau$-equivariant finite-to-one map $X\to \mathbb{P}^1$.
Since the induced map on $\mathbb{P}^1$ has a fixed point, we see that $\tau$ must permute the preimage of this fixed point, and $\tau^m$ has a fixed point for some $m\ge 1$.  
If $X$ has genus one then $X$ is an elliptic curve, and we may choose the identity to be a fixed point of $\tau^m$.
So $\tau^m$ is necessarily a group automorphism of $X$ and hence finite order, a contradiction.

Thus $X=\mathbb P^1$ and $L = \mathbb{C}(h)$ for some $h\in L$ transcendental over~$\CC$.
Moreover, we may choose
$h$ so that $\sigma^n(h)$ is either $ch$ with $c$ nonzero or $\sigma^n(h)= h+1$ (depending on whether the matrix in $\pgl_2(\CC)$ representing~$\sigma^n$ has distinct eigenvalues or not).
We claim that the former case cannot hold. Indeed, if $\sigma^n(h)=ch$ for some nonzero $c$, and writing $t=F(h)$ for some nonconstant rational function~$F$ over~$\mathbb{C}$, we see that $t+n=\sigma^n(F(h))=F(ch)$.  
Let~$\delta$ be the derivation of $L$ induced by differentiation with respect to $h$. Then applying $\delta$ to the equations $t=F(h)$ and $t+n=F(ch)$ we see that
$F'(h)=\delta(t) =F'(ch)c$.
In particular, $F'(h)h$ is fixed by $\sigma^n$ and is thus in~$\mathbb{C}$.  Since $F$ itself is nonconstant, we see that $F'(h)=\alpha/h$ for some nonzero complex number~$\alpha$, which is impossible since $F$ is a rational function.  

It follows that $\sigma^n(h)=h+1$ and so, as in the preceding case, we can write $t=F(h)$ and $t+n=F(h+1)$ for some nonconstant rational function~$F$.
Applying the derivation $\delta$, as before, gives
$F'(h) = \delta(t) = F'(h+1)$, so that $F'(h)$, being fixed by $\sigma^n$, is constant. It follows that $F(h)=\alpha h+\beta$ for some complex numbers $\alpha,\beta$ with $\alpha$ nonzero.
But this means that $\mathbb{C}(t)=\mathbb{C}(h)=L$ and so $u\in \mathbb{C}(t)$, as claimed.
\end{proof}

Now, let $\phi$ be the automorphism of the projective line that takes $[x:y]$ to $[y:x+ty]$.
Since $(\mathbb P^1)^\sigma=\mathbb P^1$, we do have that $(\mathbb P^1,\phi)$ is an inversive $\sigma$-variety over $(K,\sigma)$.
We will show that it is wild.
As an element of $\pgl_2$, the automorphism~$\phi$ is represented by the matrix
$$
A:=
\begin{bmatrix}
0&1\\
1&t
\end{bmatrix}.
$$
Moreover, for each $n\geq 1$, $\phi^{(n)}$ in $\pgl_2$ is represented by
$$B_n:=A^{\sigma^{n-1}}A^{\sigma^{n-2}}\cdots A^\sigma A.$$
Let $a_n,b_n,c_n,d_n\in \mathbb C[t]$ be such that
$
B_n=
\begin{bmatrix}
a_n&b_n\\
c_n&d_n
\end{bmatrix}
$.
A simple induction gives the recurrences:
$$a_{n+1} = c_n,  b_{n+1} = d_n, c_{n+1} = a_n + (t+n) c_n,    d_{n+1} =  b_n + (t+n)d_n$$ for $n\ge 1$, which in turn
imply the following degrees for the polynomial entries:
$\deg(a_n) = n-2, \deg(b_n)= \deg(c_n) = n-1,$ and $\deg(d_n) = n$ for all $n\ge 1$, where we take the degree of the zero polynomial to be $-1$.

\begin{lemma}
\label{norat}
Let $n\ge 1$. 
There is no rational function $u(t)$ for which we have $B_n \cdot [u:1] = [\sigma^n(u):1]$.
\end{lemma}

\begin{proof}
Suppose, towards a contradiction, that we have such a rational function $u(t)$.
Then $u(t)$ is necessarily nonzero and we write $u(t)=P(t)/Q(t)$ with $P, Q$ nonzero coprime polynomials. We have the equation:
\begin{equation*} P(t+n)(c_n P(t) + d_n Q(t)) = Q(t+n)(a_n P(t) + b_n Q(t)).
\end{equation*}
Since $P$ and $Q$ are coprime, $P(t+n)$ must divide $a_nP(t)+b_nQ(t)$ and $Q(t+n)$ divides $c_n P(t)+d_nQ(t)$; hence
 there is a nonzero polynomial $e(t)$ such that
\begin{equation*} 
e(t) \begin{bmatrix} P(t+n) \\ Q(t+n) \end{bmatrix} = \begin{bmatrix} a_n(t) & b_n(t) \\ c_n(t) & d_n(t) \end{bmatrix} \begin{bmatrix} P(t) \\ Q(t) \end{bmatrix}
\end{equation*}
Multiplying by the adjugate matrix, we obtain:
\begin{equation*}
e(t) \begin{bmatrix} d_n(t) & -b_n(t) \\ -c_n(t) & a_n(t) \end{bmatrix} \begin{bmatrix} P(t+n) \\ Q(t+n) \end{bmatrix} = \det(B_n) \begin{bmatrix} P(t) \\ Q(t) \end{bmatrix}
\end{equation*}
Since $\det(A) = -1$, we have $\det(B_n) = (-1)^n$. Because $e(t)$ must divide both $(-1)^n P(t)$ and $(-1)^n Q(t)$, and $P, Q$ are coprime, $e(t) = \gamma \in \mathbb{C}^*$. Hence we obtain the equations
\begin{align} 
\gamma P(t+n) &= a_n P(t) + b_n Q(t)\\
\gamma Q(t+n) &= c_n P(t) + d_n Q(t) \label{eq:Q_step} \\
(-1)^n \gamma^{-1} P(t) &= d_n P(t+n) - b_n Q(t+n) \label{eq:P_inv} \\
(-1)^n \gamma^{-1} Q(t) &= -c_n P(t+n) + a_n Q(t+n)
\end{align}
We now show that there are no solutions to this system by looking at two cases.
First, if $\deg P \ge \deg Q$, Equation \eqref{eq:P_inv} gives that
$$ \deg(d_n P(t+n)) = \deg(b_n Q(t+n) + (-1)^n \gamma^{-1} P(t)) $$
The left side has degree $n + \deg P$, while the right side has degree at most $$\max(n-1 + \deg P, \deg P),$$ a contradiction for $n \ge 1$.

Next, if $\deg P < \deg Q$, Equation \eqref{eq:Q_step} gives:
$$ \deg(d_n Q(t)) = \deg(\gamma Q(t+n) - c_n P(t)) $$
The left side has degree $n + \deg Q$, while the right side has degree at most
$$\max(\deg Q, n-1 + \deg P)\leq\max(\deg Q, n-2+ \deg Q),$$
again a contradiction.
\end{proof}

\begin{corollary}
\label{cor:p1wild}
$(\mathbb P^1,\phi)$ is wild as a $\sigma$-variety over $(K,\sigma)$.
\end{corollary}

\begin{proof}
Suppose, toward a contradiction, that $(\PP^1,\phi)$ is not wild.
Then there is a nonempty finite set  $S\subseteq \mathbb P^1(K)$ such that $\phi(S)=S^\sigma$.
In particular, there is $s\in\PP^1(K)$ and $n\geq 1$ such that $\phi^{(n)}(s)=\sigma^n(s)$.
If $s=[1:0]$ then we get 
$[1:0]=\sigma^n(s)=\phi^{(n)}(s)=[a_n:c_n]$,
which contradicts the fact that $c_n$ has degree $n-1$ for all $n\ge 1$.
Hence, we must have that $s$ is of the form $[u:1]$ for some $u\in K=\CC(t)^{\alg}$.
That $\phi^{(n)}(s)=\sigma^n(s)$ implies that $B_n \cdot [u:1] = [\sigma^n(u):1]$.
In other words,
$\sigma^n(u) = \frac{a_n u+b_n}{c_nu+d_n}$.
In particular, $\sigma^n(u)\in\CC(t,u)$, and Lemma~\ref{rat} implies that $u\in \CC(t)$.
But that contradicts Lemma~\ref{norat}.
\end{proof}

Note that $(\PP^1,\phi)$ is isotrivial (with binding group $\pgl_2$).
In fact, every automorphism of $\PP^1$ is isotrivial:

\begin{lemma}
Suppose $\phi$ is an automorphism of $\PP^1$ over an algebraically closed field~$K$, and $\sigma$ is an automorphism of~$K$.
Then $(\PP^1,\phi)$ is isotrivial as a $\sigma$-variety over $(K,\sigma)$.
\end{lemma}

\begin{proof}
An argument for this appears, for example, in the proof of~\cite[Proposition~7.12]{qfnmdeg}, that we repeat here.
Work in a model $(\U,\sigma)$ of $\acfa_0$ extending $(K,\sigma)$, and fix three distinct points $u_1,u_2,u_3$ of $(\PP^1,\phi)^\sharp$.
Let $A_{u_1,u_2,u_3}$ denote the unique element of $\pgl_2(\U)$ that sends $u_1,u_2,u_3$ to $0,1,\infty$, respectively.
(We are using that the action of $\pgl_2$ on $\PP^1$ is uniquely $3$-transitive.)
Then, as $\phi(u_i)=\sigma(u_i)$ for each $i=1,2,3$, we have that 
 $A_{u_1,u_2,u_3}=A_{\sigma(u_1),\sigma(u_2),\sigma(u_3)}\phi=A^\sigma_{u_1,u_2,u_3}\phi$ as elements of $\pgl_2(\U)$.
 That is, $A_{u_1,u_2,u_3}$ is an equivariant isomorphism from $(\PP^1,\phi)$ to $(\PP^1,\id)$ over the difference field $(K(u_1,u_2,u_3),\sigma)$.
\end{proof}

So, even restricting to isotrivial $\sigma$-varieties, the nonautonomous extension of the wild automorphism conjecture fails.

\bigskip
\section{Homogeneity}

\noindent
Isotriviality does, however, even in the nonautonomous case, at least in the absence of nonconstant invariant rational functions, imply homogeneity of the $\sigma$-variety, in a sense that we make precise below.
That this is true birationally is part of the binding group theorem of~\cite{qfint}, but we will show here that in the presence of wildness it works in the category of regular maps.

Fix an algebraically closed characteristic~$0$ inversive difference field $(K,\sigma)$.

\begin{definition}
By a {\em $\sigma$-group} over $(K,\sigma)$ we mean an algebraic group~$G$ over~$K$, equipped with an isomorphism $\rho:G\to G^\sigma$ of algebraic groups.
We say that a $\sigma$-variety $(X,\phi)$ is {\em homogeneous for $(G,\rho)$} if there is a  faithful transitive algebraic group action $\theta:G\times X\to X$ which is equivariant for $\rho\times\phi$ on $G\times X$ and $\phi$ on $X$.
This latter condition means that
\begin{equation*}
\xymatrix{
G\times X\ar[rr]^\theta\ar[d]_{(\rho\times\phi)} && X\ar[d]^\phi\\
G^\sigma\times X^\sigma\ar[rr]_{\theta^\sigma} && X^\sigma
}
\end{equation*}
commutes.
We call $(X,\phi)$ {\em homogeneous} if it is homogeneous for some $\sigma$-group.
\end{definition}

\begin{remark}
Since we allow the possibility that~$G$ is not connected, there is some inconsistency in calling $(G,\rho)$ a $\sigma$-group; after all, it is not technically even a $\sigma$-variety.
However, $\rho$ will restrict to an isomorphism $G^0\to(G^0)^\sigma$, where $G^0$ is the connected component of identity of~$G$, and so $(G^0,\rho|_{G^0})$ is an inversive $\sigma$-variety.
\end{remark}

Actually we will need a weaker notion:

\begin{definition}
We say that a $\sigma$-variety $(X,\phi)$ is {\em pre-homogeneous for} a $\sigma$-group $(G,\rho)$ over $(K,\sigma)$ if there is a  rational map $\theta:G\times X\dto X$ such that:
\begin{itemize}
\item
$G$ is an algebraic group of birational transformations of~$X$.
This means that $\theta(g,-)=:\theta_g$ is a birational transformation of~$X$ for each $g\in G$,  and that $g\mapsto\theta_g$ is an injective group homomorphism from $G$ into $\bir(X)$.
\item
The action is generically transitive.
This means that there is $a\in X(\U)$ Zariski generic over~$K$ such that
$Ga:=\{\theta_g(a):g\in G, \theta_g\text{ defined at }a\}$
is Zariski dense in $X$.
\item
$\theta$ is equivariant in the sense that
\begin{equation}
\label{phirho}
\xymatrix{
G\times X\ar@{-->}[rr]^\theta\ar[d]_{(\rho\times\phi)} && X\ar[d]^\phi\\
G^\sigma\times X^\sigma\ar@{-->}[rr]_{\theta^\sigma} && X^\sigma
}
\end{equation}
commutes.
\end{itemize}
Again, $(X,\phi)$ is {\em pre-homogeneous} if it is pre-homogeneous for some $\sigma$-group.
\end{definition}

\begin{remark}
Note that if $(X,\phi)$ is pre-homogeneous for $(G,\rho)$ then $(X,\phi)$ is pre-homogeneous for $(G^0,\rho|_{G^0})$, where $G^0$ is the connected component of identity.
\end{remark}

\begin{example}
\label{wildiso-prehom}
If $(X,\phi)$ is an isotrivial $\sigma$-variety over $(K,\sigma)$ that admits no nonconstant invariant rational functions then it is pre-homogeneous.
\end{example}

\begin{proof}
As mentioned earlier, it is shown in~\cite{qfint} that isotriviality alone gives rise to a binding group $(G,\rho)$ for $(X,\phi)$, satisfying all the properties of pre-homogeneity above except that the action need not be generically transitive.
However, by~\cite[Proposition~3.4 and Lemma~5.5]{qfint}, as there are no nonconstant invariant rational functions, $(G,\rho)^\sharp$ will act transitively on the generic points of $(X,\phi)$.
As these are Zariski dense in~$X$, the action of~$G$ on~$X$ is generically transitive.
\end{proof}

Suppose $(X,\phi)$ is a projective inversive $\sigma$-variety.
How far is pre-homogeneity from homogeneity?
Weil's regularisation theorem comes close to bridging the gap.
By a {\em projective regularisation} of $(X,\phi)$ we mean a projective variety $\widehat X$ over~$K$, a birational map $\widehat\phi:\widehat X\dto\widehat X^\sigma$, and birational map $\pi:X\dto\widehat X$ such that:
\begin{itemize}
\item $\pi$ regularises the action of~$G$ in the sense of Weil~\cite{weil}.
That is, there is a regular algebraic group action 
$\widehat\theta:G\times\widehat X\to \widehat X$ such that
\begin{equation}
\label{weil}
\xymatrix{
G\times X\ar@{-->}[rr]^\theta\ar@{-->}[d]_{(\id\times\pi)} && X\ar@{-->}[d]^\pi\\
G\times \widehat X\ar[rr]_{\widehat\theta} && \widehat X
}
\end{equation}
commutes.
\item
$\pi:(X,\phi)\dto (\widehat X,\widehat\phi)$ is a birational equivalence of $\sigma$-varieties.
That is, the following diagram commutes:
\begin{equation}
\label{diag:hatphi}
\xymatrix{
X\ar[rr]^\phi\ar@{-->}[d]_\pi && X^\sigma\ar@{-->}[d]^{\pi^\sigma}\\
\widehat X\ar@{-->}[rr]_{\widehat\phi} && \widehat X^\sigma
}
\end{equation}
\end{itemize}
A projective regularisation always exists by Weil's theorem; see~\cite{weil}, or rather~\cite{zaitsev} in this setting where~$G$ need not be connected.
Actually, Weil's theorem gives us~(\ref{weil}) and then we use~$\pi$ to define~$\widehat\phi$ so that~(\ref{diag:hatphi}) holds.
The trade-off is that while $\widehat\theta$ is now a regular (as opposed to only rational) group action, the ``dynamics" $\widehat\phi$ is, in principle, no longer an isomorphism but just a birational map.
However, in the wild case such regularisations are much better behaved:

\begin{proposition}
\label{prop:wildreg}
A wild projective inversive $\sigma$-variety that is pre-homogeneous for $(G,\rho)$ is homogeneous for $(G,\rho)$.
\end{proposition}

\begin{proof}
Suppose $(X,\phi)$ is a wild projective inversive $\sigma$-variety over $(K,\sigma)$, and
let $\pi:(X,\phi)\dto (\widehat X,\widehat\phi)$ be a projective regularisation.
We will show that $(\widehat X,\widehat\phi)$ is a $\sigma$-variety (that is, $\widehat\phi$ is an isomorphism), that it is homogeneous for $(G,\rho)$, and that~$\pi$ is an isomorphism of $\sigma$-varieties.
This will imply that $(X,\phi)$ is itself homogeneous for $(G,\rho)$, as desired.

From~(\ref{phirho}), (\ref{weil}), and~(\ref{diag:hatphi}), we immediately deduce that
\begin{equation}
\label{hatphirho}
\xymatrix{
G\times \widehat X\ar[rr]^{\widehat\theta}\ar@{-->}[d]_{(\rho\times\widehat\phi)} && \widehat X\ar@{-->}[d]^{\widehat\phi}\\
G^{\sigma}\times \widehat X^{\sigma}\ar[rr]_{\widehat\theta^{\sigma}} && \widehat X^{\sigma}
}
\end{equation}
commutes.
The generic transitivity condition in pre-homogeneity implies that the action of~$G$ on $\widehat X$ is also generically transitive.
It follows (as $K$ is algebraically closed) that there is $a\in\widehat X(K)$ with $Ga=:U$ a Zariski dense and open subset of~$\widehat X$.

\begin{claim}
\label{U}
$\widehat\phi$ restricts to an isomorphism from $U$ to $U^{\sigma}$.
\end{claim}

\begin{claimproof}
Let~$I\subseteq \widehat X$ denote the indeterminacy locus of $\widehat\phi$, and suppose $b\in U$.
As $U\not\subseteq I$, and recalling that $U$ is a $G$-orbit, there must be some $g\in G$ such that $gb\in U\setminus I$.
Hence $\widehat\phi$ is defined at $gb$.
It follows from~(\ref{hatphirho}) that $\widehat\phi=\rho(g)^{-1}\circ\widehat\phi\circ g$ is defined at~$b$.
This shows that $\widehat\phi$ is regular on~$U$.

Next we observe that $\widehat\phi(U)=U^\sigma$.
Indeed, as $\widehat\phi(ga)=\rho(g)\widehat\phi(a)$ we have that $\widehat\phi(U)=G^\sigma\cdot\widehat\phi(a)$.
In particular,  $G^\sigma\cdot\widehat\phi(a)$ is a Zariski dense orbit in $\widehat X^\sigma$.
On the other hand,  $U^\sigma=G^\sigma\cdot\sigma(a)$ is also a Zariski dense orbit in $\widehat X^\sigma$.
Hence $G^\sigma\cdot\widehat\phi(a)=G^\sigma\cdot\sigma(a)$, and $\widehat\phi(U)=U^\sigma$, as desired.

A similar argument works for $\widehat\phi^{-1}$, showing that $\widehat\phi^{-1}$ is regular on $U^\sigma$ and $\widehat\phi^{-1}(U^\sigma)=U$.
We thus have regular maps $\widehat\phi_U:U\to U^\sigma$ and $(\widehat\phi^{-1})_{U^\sigma}:U^\sigma\to U$ which are inverse to each other as rational maps, and hence as regular maps.
\end{claimproof}

We now argue that $U=\widehat X$.
Indeed, it follows from Claim~\ref{U} that $(U,\widehat\phi|_U)$ is an inversive $\sigma$-variety.
As $U$ is a Zariski open and dense subset of~$\widehat X$, we can consider the birational map $\pi_U:X\dto U$ obtained by restricting $\pi$, and we note that $\pi_U:(X,\phi)\dto (U,\widehat\phi|_U)$ is then an equivariant dominant rational map of inversive $\sigma$-varieties.
As $(X,\phi)$ is wild, $\pi_U$ is regular -- this is Corollary~\ref{cor:emptyint}.
As~$X$ is projective, this forces $U=\widehat X$.

Hence, by Claim~\ref{U}, $\widehat\phi:\widehat X\to\widehat X^\sigma$ is an isomorphism and $(\widehat X,\widehat\phi)$ is an honest inversive $\sigma$-variety.
Moreover, (\ref{hatphirho}) together with the fact that $Ga=U=\widehat X$, implies that $(\widehat X,\widehat\phi)$ is homogeneous for $(G,\rho)$.
We also get, from the fact that $U=\widehat X$, that $\pi=\pi_U$, and hence $\pi:X\to\widehat X$ is regular.
So $\pi:(X,\phi)\to(\widehat X,\widehat\phi)$ is a (surjective) morphism of $\sigma$-varieties.

It remains only to argue that~$\pi$ is an isomorphism.
Since~$\pi$ is birational and $\widehat X$ is normal (it is a homogeneous space for an algebraic group), this will follow from Zariski's Main Theorem once we prove that all the fibres of $\pi$ are finite, see~\cite[Section~III.9]{mumford}.
Now, the set of points~$E$ in $\widehat X$ over which the fibre is positive-dimensional is a proper closed subvariety.
The equivariance of~$\pi$ implies that~$E$ is invariant.
But $(\widehat X,\widehat\phi)$ is wild because
the pull-back under~$\pi$ of any proper invariant subvariety of $(\widehat X,\widehat\phi)$ would be a proper invariant subvariety of $(X,\phi)$.
Hence~$E$ is empty, as desired.
\end{proof}

\begin{corollary}
Every isotrivial wild projective inversive $\sigma$-variety is homogeneous.
\end{corollary}

\begin{proof}
Suppose $(X,\phi)$ is a wild projective inversive $\sigma$-variety over $(K,\sigma)$.
By Lemma~\ref{nirfs}, $(X,\phi)$ admits no nonconstant invariant rational functions.
If $(X,\phi)$ is in addition isotrivial then, by Example~\ref{wildiso-prehom}, it is pre-homogeneous.
Now apply Proposition~\ref{prop:wildreg}.
\end{proof}

\bigskip
\section{Fundamental isotriviality}
\label{sec:fi}

\noindent
The gap between homogeneity and actually being an abelian variety, for a projective $\sigma$-variety, is bridged if the group action is not only transitive but free.
To ensure freeness of the action, we need a stronger condition than isotriviality.

Fix an algebraically closed inversive difference field $(K,\sigma)$ of characteristic~$0$.

\begin{definition}
\label{def:fund}
A $\sigma$-variety $(X,\phi)$ over $(K,\sigma)$ is {\em fundamental isotrivial} if there are algebraically independent generic points $a,b$ of $(X,\phi)$ and a finite tuple~$c$ from $\C=\fix(\U,\sigma)$, such that $K(a,b)=K(a,c)$.
\end{definition}

Fundamental isotriviality says that the base extension of $(K,\sigma)$ required to witness isotriviality can be taken to be the function field of the $\sigma$-variety itself.
It is inspired by what is often called ``fundamental internality" in model theory.

We can improve upon the definition:

\begin{lemma}
\label{lem:fund}
Suppose $(X,\phi)$ is fundamental isotrivial.
For all generic points $a,b$ of $(X,\phi)$ there is a finite tuple~$c$ from $\C$, such that $K(a,b)=K(a,c)$.
\end{lemma}

\begin{proof}
By definition there are algebraically independent generic points $a',b'$ of $(X,\phi)$ and a finite tuple~$c'$ from $\C$, such that $K(a',b')=K(a',c')$.
We argue model-theoretically, but just for convenience.
First consider the case when~$a$ and~$b$ are algebraically independent (over~$K$).
Then, by stationarity of quantifier-free types over $\acl$-closed sets, $\qftp(a,b/K)=\qftp(a',b'/K)$.
So $(a,b)\mapsto(a',b')$ induces an isomorphism $(K(a,b),\sigma)\to(K(a',b'),\sigma)$ of difference field extensions of $(K,\sigma)$.
Hence  there is~$c$ from $\fix(K(a,b),\sigma)$, corresponding to $c'$, such that 
$K(a,b)=K(a,c)$.
In general, given generics $a,b$ of $(X,\phi)$ not necessarily independent, existence of nonforking extensions for quantifier-free types gives us a generic point~$\widehat b$ such that $\qftp(a,b/K)=\qftp(a,\widehat b/K)$, and $a, \widehat b$ are algebraically independent.
Again $(a,b)\mapsto(a,\widehat b)$ induces an isomorphism $(K(a,b),\sigma)\to(K(a,\widehat b),\sigma)$ and the result for $a, \widehat b$ thus transfers to $a,b$ (for a different choice of~$c$).
\end{proof}

Translations on algebraic groups (Example~\ref{ex:translation}) are fundamental isotrivial.
In fact, in the autonomous case, isotriviality almost always agrees with fundamental isotriviality:

\begin{proposition}
\label{prop:iso-fund}
Suppose $X$ is an irreducible variety over an algebraically closed field~$k$, and $\phi\in\aut(X)$ admits no nonconstant invariant rational functions.
If $(X,\phi)$ is isotrivial over $(k,\id)$ then it is fundamental isotrivial.
\end{proposition}

\begin{proof}
As we have already mentioned, autonomous isotrivial dynamics are translational by~\cite[Theorem~5.1]{qfint}.
That is, there is an algebraic group~$G$ over~$k$ acting faithfully on~$X$ such that $\phi$ is translation by some $g_0\in G(k)$.
Replacing~$G$ by the Zariski closure of the group generated by~$g_0$, we may assume that~$G$ is commutative.
The fact that $(X,\phi)$ admits no invariant rational functions implies that there is $x\in X(k)$ whose $\phi$-orbit is Zariski dense in~$X$.
(This is the truth of the Zariski dense orbit conjecture for translational dynamics, see for example~\cite[Corollary~5.2]{qfint}.)
It follows that $Gx=:U$ is Zariski open in $X$, and is itself $G$-invariant by commutativity.
In particular, $(U,\phi)$ is also an inversive $\sigma$-variety over $(k,\id)$, and, since fundamental isotriviality is a birational invariant, it suffices to show that $(U,\phi)$ is fundamental isotrivial.
Note that $G$ acts transitively and faithfully on~$U$, the latter as $U$ is Zariski dense in~$X$.
Now fix a generic point~$a$ of $(X,\phi)$ and let $L=k(a)$.
Note that $S:=\Stab(x)=\Stab(a)$ by commutativity of~$G$ and transitivity of the action.
Hence we have an $L$-definable isomorphism $\rho:U\to G/S$ given by $\rho(u)=hS$ if $ha=u$.
(Here we eschew the subscripts indicating that we have taken base extensions to~$L$.)
Note that $\rho^\sigma:U\to G/S$ is given by $\rho^\sigma(u)=hS$ if $h(g_0a)=u$, since $\sigma(a)=\phi(a)=g_0a$.
Hence,
\begin{eqnarray*}
\rho^\sigma\phi(u)=hS
&\iff&
h(g_0a)=\phi(u)\\
&\iff&
h(g_0a)=g_0u\\
&\iff&
ha=u\\
&\iff&
\rho(u)=hS
\end{eqnarray*}
That is, $\rho^\sigma\phi=\rho$, proving that 
$\rho:(U,\phi)\to(G/S,\id)$
is equivariant, and hence $(U,\phi)$ is fundamental isotrivial.
\end{proof}

Beyond the autonomous case, this does not hold.
Indeed the isotrivial $\sigma$-variety structure on the projective line from Corollary~\ref{cor:p1wild}, which does not admit nonconstant invariant rational functions (as it is wild), is not fundamental isotrivial.
In fact, in the fundamental isotrivial case only abelian varieties can be wild:

\begin{proposition}
\label{prop:fund-wac}
Suppose $(X,\phi)$ is a projective inversive $\sigma$-variety over $(K,\sigma)$.
If $(X,\phi)$ is fundamental isotrivial and wild then $X$ is an abelian variety.
\end{proposition}

\begin{proof}
Let $(G,\rho)$ be the binding group of $(X,\phi)$ from~\cite{qfint}.
Wildness ensures that $(X,\phi)$ admits no nonconstant invariant rational functions (Lemma~\ref{nirfs}), and so $(X,\phi)$ is pre-homogeneous for $(G,\rho)$ by Example~\ref{wildiso-prehom}.
By Proposition~\ref{prop:wildreg}, $(X,\phi)$ is homogeneous for $(G,\rho)$.

By Lemma~\ref{isolate} wildness implies that all elements of $(X,\phi)^\sharp$ are generic.
So $(G,\rho)^\sharp$ acts faithfully on $(X,\phi)^\sharp$.
We claim that that action is free.
That is, if $a\in(X,\phi)^\sharp$ and $g\in(G,\rho)^\sharp$ with $ga=a$ then $g=\id$.
Indeed, for any $b\in(X,\phi)^\sharp$, Lemma~\ref{lem:fund} gives a tuple~$c$ from $\C=\fix(\U,\sigma)$ and a rational function $f(x,y)$ such that $b=f(a,c)$.
The binding group preserves all algebraic relations between generic points of $(X,\phi)$ and fixed points.
That is, we have $gb=f(ga,c)=f(a,c)=b$.
As~$b$ was arbitrary, we have that $g=\id$, as desired.

Now, fix $a\in(X,\phi)^\sharp$ such that $Ga=X$.
By freeness, $\Stab(a)\leq G$ has no nontrivial points in $(G,\rho)^\sharp$.
But $\Stab(a)$, and hence its connected component of identity, is an invariant subgroup of $(G,\rho)$, and so should have a Zariski dense set of $\sharp$-points.
That is, $\Stab(a)$ is finite.
Since $G/\Stab(a)$ is isomorphic to the projective variety~$X$, we must have that~$G$ itself is a projective algebraic group.
In particular it is commutative, and so $X= G/\Stab(a)$ inherits the structure of a projective algebraic group.
As~$X$ is irreducible, it is an abelian variety.
\end{proof}

\bigskip
\section{Compound isotriviality}
\label{sec:ci}

\noindent
We now return to the autonomous case of an algebraically closed field~$k$, of characteristic~$0$, equipped with the trivial automorphism.
We wish to define a class of automorphisms of algebraic varieties that are built up by fibrations from isotrivial inversive $\sigma$-varieties (over difference field extensions of~$k$).

\begin{definition}
\label{compiso}
Suppose~$X$ is an irreducible variety over~$k$ and $\phi\in\aut(X)$.
We say that $(X,\phi)$ is {\em compound isotrivial} if it admits a sequence of equivariant dominant rational maps,
$(X,\phi)\dto (X_1,\phi_1)\dto\cdots\dto(X_n,\phi_n)$ such that
\begin{itemize}
\item[(i)]
each $\phi_i\in\aut(X_i)$,
\item[(ii)]
the generic fibre of each of the equivariant rational maps is absolutely irreducible and isotrivial as an inversive $\sigma$-variety in its own right, and
\item[(iii)]
$(X_n, \phi_n)$ is isotrivial.
\end{itemize}
If we have fundamental isotriviality in~(ii) and~(iii), then we say that $(X,\phi)$ is {\em compound fundamental isotrivial}.
\end{definition}

\begin{remark}
Condition~(i) is somewhat unfortunate, it would be better to only ask that each $\phi_i$ be a birational transformation, or even just a dominant rational self-map.
But we require automorphisms in our proof of Theorem~\ref{wild-compoundiso} below.
\end{remark}

\begin{example}
Fix an elliptic curve~$E$ over~$\CC$.
Let $X=E\times E$, and equip it with the automorphism $\phi(x,y)=(x+y,y)$.
Then the second co-ordinate projection $\pi:(X,\phi)\to(E,\id)$ is surjective and equivariant.
The generic fibres of~$\pi$ are translations on~$E$, and hence fundamental isotrivial, as in Example~\ref{ex:translation}.
So $\pi$ witnesses the compound fundamental isotriviality of $(X,\phi)$.
On the other hand, $(X,\phi)$ is not itself isotrivial: if it were it would be translational and some iterate of $\phi$ would have to be in $\aut^0(X)$, which is just $X$ itself acting by translation.
But clearly no iterate of $\phi$ acts by translation.

We can, similarly, get wild examples by fixing a nontorsion point $p\in E(\CC)$ and defining $\phi$ on $X=E\times E$ by $\phi(x,y)=(x+y,y+p)$, so that $(X,\phi)\to (E,y\mapsto y+p)$ witnesses the compound isotriviality of the nonisotrivial $(X,\phi)$.
\end{example}

But here is a general way of producing compound isotrivial automorphisms on abelian varieties.
When $A$ is an abelian variety, we should distinguish between automorphisms of the {\em variety}~$A$, which we will denote by $\autvar(A)$, and automorphisms of the {\em algebraic group}~$A$, which we will denote by $\autgrp(A)$.
Of course, every element of $\autvar(A)$ is obtained from an element of $\autgrp(A)$ by composing with a translation.

\begin{proposition}
\label{prop:av-ci}
Suppose~$A$ is an abelian variety over~$k$ and $\phi\in\autvar(A)$ is of the form $\alpha+a$ where $\alpha\in\autgrp(A)$ and $a\in A(k)$.
If~$\alpha$ is unipotent then $(A,\phi)$ is compound fundamental isotrivial.
\end{proposition}

\begin{proof}
We proceed by induction on the degree of unipotency.
That is, we proceed by induction on~$n\geq 1$ where $(\alpha-\id)^n=0$.
When $n=1$, $\alpha=\id$ and $\phi$ is translation, so that $(A,\phi)$ is fundamental isotrivial.

Let $\beta:=\alpha-\id$, and suppose that $\beta^{n+1}=0$.
Consider the abelian variety $B:=A/\beta^n(A)$, with quotient map $\pi:A\to B$.
Since $\alpha$ commutes with $\beta$, it induces an automorphism $\overline\alpha\in\autgrp(B)$.
Namely,
$\overline\alpha\circ\pi=\pi\circ \alpha$.
Letting $\overline\phi:=\overline\alpha+\pi(a)\in\autvar(B)$, we have the surjective equivariant morphism
$\pi:(A,\phi)\to(B,\overline\phi)$.
Note that 
$(\overline\alpha-\id)^n\circ\pi=\pi\circ(\alpha-\id)^n=\pi\circ\beta^n=0$.
As $\pi$ is surjective, this forces $(\overline\alpha-\id)^n=0$ on~$B$.
So, by the inductive hypothesis, $(B,\overline\phi)$ is compound fundamental isotrivial.

It remains to show that the generic fibre of $\pi:(A,\phi)\to(B,\overline\phi)$ is fundamental isotrivial.
Let $b\in (A,\phi)^\sharp$ be such that $\overline b=\pi(b)$ is generic in $(B,\overline\phi)$.
The restriction of $\phi$ gives us an inversive $\sigma$-variety structure on $A_{\overline b}$ over $K:=k(\overline b)^{\alg}$ which we denote by $(A_{\overline b},\phi)$.
We claim that $(A_{\overline b},\phi)$ is isomorphic to $(\beta^n(A),\id)$ over~$L:=K(b)$, which will suffice.
Note that $A_{\overline b}$ is just the coset $b+\beta^n(A)$, so that we do have an isomorphism $\rho:A_{\overline b}\to \beta^n(A)$ over~$L$ given by $\rho(x)=x-b$.
Now
$$(A_{\overline b})^\sigma=A_{\sigma(\overline b)}=A_{\overline\phi(\overline b)}=A_{\pi(\phi(b))}=\phi(b)+\beta^n(A)$$
and $\rho^\sigma:(A_{\overline b})^\sigma\to\beta^n(A)$ is given by 
$$\rho^\sigma(x)=x-\sigma(b)=x-\phi(b).$$
For  $x\in A_{\overline b}$, we compute:
\begin{eqnarray*}
\rho^\sigma\phi(x)
&=&
\rho^\sigma\phi(\rho(x)+b)\\
&=&
\rho^\sigma\big(\alpha(\rho(x)+b)+a\big)\\
&=&
\rho^\sigma\big(\alpha(\rho(x))+\alpha(b)+a\big)\\
&=&
\rho^\sigma\big(\rho(x)+\alpha(b)+a\big)\ \ \ \text{ as $\alpha=\id$ on $\beta^n(A)$ since $\beta^{n+1}=0$}\\
&=&
\rho^\sigma\big(\rho(x)+\phi(b)\big)\\
&=&
\rho(x).
\end{eqnarray*}
That is, $\rho^\sigma\phi=\rho$.
Hence $\rho:(A_{\overline b},\phi)\to(\beta^n(A),\id)$ is equivariant, as desired.
\end{proof}

\begin{corollary}
\label{cor:avwild-ci}
If $\phi$ is a wild automorphism of an abelian variety~$A$ over~$k$ then $(A,\phi)$ is compound fundamental isotrivial.
\end{corollary}

\begin{proof}
Theorem~7.2 of~\cite{rrz}, which characterises the wild automorphisms of abelian varieties, tells us that if $\phi$ is wild then $\phi=\alpha+a$ for some unipotent $\alpha\in\autgrp(A)$ and $a\in A(k)$.
Now apply Proposition~\ref{prop:av-ci}.
\end{proof}

It therefore makes sense, in working toward the wild automorphism conjecture, to first restrict our attention to compound isotrivial, or even compound fundamental isotrivial, algebraic dynamics.
We settle the latter case:

\begin{theorem}
\label{wild-compoundiso}
Suppose $X$ is an irreducible projective variety over an algebraically closed field~$k$ of characteristic zero, and $\phi\in \aut(X)$ is wild.
If $(X,\phi)$ is compound fundamental isotrivial then $X$ is an abelian variety.
\end{theorem}

\begin{proof}
We work in an ambient difference-closed field $(\U,\sigma)$ with $k\subseteq\fix(\U,\sigma)$, and proceed by induction on the number of steps in the sequence witnessing the compound fundamental isotriviality.
In the base case $(X,\phi)$ is isotrivial, and hence translational.
But the wild automorphism conjecture holds for translational dynamics by~\cite[Proposition~3.2]{rrz}.

For the inductive step, consider 
$\pi:(X,\phi)\dto(X_1,\phi_1)$, the first equivariant dominant rational map in a sequence witnessing compound fundamental isotriviality.
As $(X,\phi)$ is wild, $\pi$ is regular by Corollary~\ref{cor:emptyint}.
Hence $X_1$ is also projective, and $\pi$ is surjective, so that $(X_1,\phi_1)$ is wild.
By the induction hypothesis, $X_1$ is an abelian variety.
Let $b\in (X_1,\phi_1)^\sharp$ be Zariski generic over~$k$.
So $K:=k(b)^{\alg}$ is an inversive difference subfield of $(\U,\sigma)$, and we consider the (nonautonomous) projective inversive $\sigma$-variety $(X_b,\phi)$ over  $(K,\sigma)$.

\begin{claim}
\label{wildfibre}
$(X_b,\phi)$ is wild.
\end{claim}

\begin{claimproof}
Suppose not, and that $Y$ is a proper invariant subvariety over~$K$.
One complication is that $Y$ may not be irreducible.
But, by $\phi$-invariance, we do have that $\phi$ will map each irreducible component of $Y$ to the $\sigma$-transform of some other irreducible component.
It follows that $\sigma^{-1}\circ\phi$ permutes the irreducible components of $Y$, and hence some iterate of it, say $(\sigma^{-1}\circ\phi)^n$ preserves an irreducible component, say $Y_\circ$.
Now, $(\sigma^{-1}\circ\phi)^n=\sigma^{-n}\circ\phi^{(n)}$.
So $\phi^{(n)}$ maps $Y_\circ$ to $(Y_\circ)^{\sigma^n}$.
That is, $Y_\circ$ is invariant for the $\sigma^n$-variety $(X_b,\phi^{(n)})$.
By irreducibility, and existential closedness of $(\U,\sigma)$, there is a Zariski generic point $e\in Y_\circ(\U)$ over~$K$ such that $\sigma^n(e)=\phi^{(n)}(e)$.
Now, let $Z:=\loc(e/k)$.
Since $Y\neq X_b$, and~$b$ is Zariski generic in~$X_1$ over~$k$, we have that~$Z\neq X$.
On the other hand, $\phi^{(n)}(e)=\sigma^n(e)$ implies that $Z$ is invariant for $(X,\phi^{(n)})$.
But wildness of $\phi$ implies wildness of $\phi^{(n)}$, by Lemma~\ref{iterate-wild}, which precludes the existence of such a~$Z$.
\end{claimproof}

By assumption, $(X_b,\phi)$ is fundamental isotrivial.
Hence wildness implies $X_b$ is an abelian variety, by Proposition~\ref{prop:fund-wac}.
Since $X$ is an algebraic fibre space over an abelian variety, with general fibre an abelian variety, we have, by a result of Cao and P\v{a}un~\cite{cp17}, that the Kodaira dimension of~$X$ is nonnegative.
But wildness rules out positive Kodaira dimension (this is~\cite[Corollary~4.2]{rrz}).
So the Kodaira dimension of~$X$ is zero.
In that case wildness implies that~$X$ is of the form $A\times Y$ where $A$ is an abelian variety and $Y$ has trivial Albanese (this is~\cite[Theorem~1.2]{kirson}).
It follows
that~$A$ is the Albanese of~$X$.
Since $\pi:X\to X_1$ is a morphism to an abelian variety, the universal property of the Albanese implies that~$\pi$ must factor through the co-ordinate projection $X\to A$.
That is, we have the following commuting diagram of surjective regular maps:
$$\xymatrix{
X=A\times Y\ \ \ \ \ \ \ar[d]\ar[dr]^\pi\\
A\ar[r]^{f\ \ \ }& X_1
}$$
for some regular map~$f$.
It follows that the fibre of~$\pi$ over~$b$ is $X_b=f^{-1}(b)\times Y$.
But $X_b$ is an abelian variety, and $Y$ has trivial Albanese, forcing $\dim Y=0$.
That is, $X=A$ is an abelian variety.
\end{proof}

Theorem~\ref{wild-compoundiso} reduces the wild automorphism conjecture to showing that {\em every wild automorphism of a projective variety is compound fundamental isotrivial}.



\end{document}